\documentclass[a4paper, 11pt]{article}
\usepackage{amsmath}
\usepackage{amssymb}
\usepackage{amsthm}
\usepackage[top=30truemm, left=13truemm, right=13truemm]{geometry}
\usepackage{amscd}
\usepackage{graphicx}
\usepackage[all]{xy}

\newtheorem{dfn}{Definition}[section]
\newtheorem{thm}[dfn]{Theorem}
\newtheorem{lem}[dfn]{Lemma}
\newtheorem{rem}[dfn]{Remark}
\newtheorem{cor}[dfn]{Corollary}

\newtheorem{prop}[dfn]{Proposition}

\newcommand\C{\mathbb{C}}
\newcommand\K{\mathbb{K}}
\newcommand\Z{\mathbb{Z}}

\newcommand\M{\mathbb{M}}

\newcommand\cO{\mathcal{O}}

\newcommand\Aut{\operatorname{Aut}}
\newcommand\End{\operatorname{End}}
\newcommand\id{\mathrm{id}}
\newcommand\Ad{\mathrm{Ad}}

\newcommand{\Tor}{\operatorname{Tor}}

\newcommand{\Diag}{\operatorname{Diag}}

\title{The group structure of the homotopy set whose target is the automorphism group of the Cuntz algebra}
\author{Masaki Izumi 
\thanks{Supported in part by JSPS KAKENHI Grant Number JP15H03623}\\
Graduate School of Science \\
Kyoto University \\
Sakyo-ku, Kyoto 606-8502, Japan 
\and Taro Sogabe \\Graduate School of Science \\
Kyoto University \\
Sakyo-ku, Kyoto 606-8502, Japan }
\begin{document}
\maketitle
\abstract
We determine the group structure of the homotopy set whose target is the automorphism group of the Cuntz algebra $\mathcal{O}_{n+1}$ 
for finite $n$ in terms of K-theory. 
We show that there is an example of a space for which the homotopy set is a non-commutative group, and hence 
the classifying space of the automorphism group of the Cuntz algebra for finite $n$ is not an H-space. 
We also make an improvement of Dadarlat's classification of continuous fields of the Cuntz algebras in terms of vector bundles.  

\section{Introduction}
Dadarlat \cite{D0} computed the homotopy set $[X,\Aut A]$ for a Kirchberg algebra $A$ under a mild assumption of a space $X$. 
He constructed a bijection between $[X,\Aut A]$ and a relevant KK-group, and showed that it is a group homomorphism 
when $X$ is an H'-space (co-H-space). 
However, the group structure of $[X,\Aut A]$ for more general $X$ is still unknown. 

The Cuntz algebra $\mathcal{O}_{n+1}$ is a typical example of a Kirchberg algebra, and it plays an important role in operator algebraic 
realization of the mod $n$ K-theory \cite{RS}. 
Dadarlat's computation shows that $[X,\Aut \cO_{n+1}]$ as a set is identified with the mod $n$ K-group $K^1(X;\Z_n)$. 
One of the main purposes of this paper is to determine the group structure of $[X,\Aut \cO_{n+1}]$, 
and we show that it is indeed different from the ordinary group structure of $K^1(X;\Z_n)$ in general. 
In particular, we verify that the group $[X,\Aut \cO_{n+1}]$ is non-commutative when $X$ is the product of 
the Moore space $M_n$ and its reduced suspension $\Sigma M_n$.  
Our computation uses the Cuntz-Toeplitz algebra $E_{n+1}$ in an essential way, for which  
the homotopy groups of the automorphism group are computed in \cite{ST}. 

The unitary group $U(n+1)$ acts on $\cO_{n+1}$ through the unitary transformations of the linear span of 
the canonical generators, and it induces a map from $[X,BU(n+1)]$ to $[X,B\Aut \cO_{n+1}]$. 
When $X$ is a finite CW-complex with dimension $d$, Dadarlat \cite[Theorem 1.6]{D1} showed that the map is a bijection provided 
that $n\geq \lceil(d-3)/2\rceil$ and $H^*(X)$ has no $n$-torsion. 
Another purpose of this paper is to remove the first condition by a localization trick. 

We use the following notation throughout the paper.  
For a unital C*-algebra $A$, we denote by $U(A)$ the unitary group of $A$, and by $U(A)_0$ the path component of $1_A$ in $U(A)$.
For a non-unital C*-algebra $B$, we denote its unitization by $B^{\sim}$.
We denote by $\mathbb{B}(H)$ the algebra of bounded operators on a Hilbert space $H$, 
by $\mathbb{K}$ the algebra of compact operators on a separable Hilbert space, 
and by $\M_n$ the algebra of $n$ by $n$ matrices.

Our standard references for K-theory are \cite{Bl, K}. 
For a projection $p\in A$ (resp. a unitary $u\in U(A)$), we denote 
denote by $[p]_0$  (resp. $[u]_1$) its class in the K-group $K_0(A)$ (resp. $K_1(A)$). 
For a compact Hausdorff space $X$, we identify the topological K-groups $K^i(X)$ with $K_i(C(X))$ where $C(X)$ is 
the C*-algebra of the continuous functions on $X$. 
When moreover $X$ is path connected, we choose a base point $x_0$, and set $\tilde{K}^i(X)$ to be the kernel of the evaluation map 
$({\rm ev}_{x_0})_* \colon K^i(X)\to K^i(\{x_0\})=\mathbb{Z}$, which is identified with $K_i(C_0(X, x_0))$ 
where $C_0(X, x_0)$ is the C*-algebra of the continuous functions on $X$ vanishing at $x_0$.
We denote by $\Sigma X$ the reduced suspension of $X$. 
For two topological spaces $X$ and $Y$, we denote by $\operatorname{Map}(X,Y)$ the set of continuous map from $X$ to $Y$, 
and by $[X,Y]$, the quotient of $\operatorname{Map}(X,Y)$ by homotopy equivalence.

\textbf{Acknowledgement. } 
The authors would like to thank Marius Dadarlat and Ulrich Pennig for stimulating discussions. 
Masaki Izumi would like to thank Isaac Newton Institute for Mathematical Sciences for its hospitality.

\section{Mod $n$ K-theory}
In this section, we summarize the basics of mod n K-theory from the view point of operator algebras. 

Recall that the Cuntz algebra $\cO_{n+1}$ is the universal C*-algebra generated by $n+1$ isometries $\{S_i\}_{i=0}^n$ with mutually orthogonal ranges  
whose summation is 1. 
Its K-groups are  
$$K_0(\mathcal{O}_{n+1})=\mathbb{Z}_n, \; K_1(\mathcal{O}_{n+1})=0,$$
(see \cite[Theorem 3.7, 3.8, Corollary 3.11]{C}). 
The Cuntz Toeplitz algebra $E_{n+1}$ is the universal C*-algebra generated by $n+1$ isometries $\{T_i\}_{i=0}^{n}$ 
with mutually orthogonal ranges, and it is KK-equivalent to the complex numbers $\C$. 
The closed two-sided ideal generated by the minimal projection $e\colon=1-\sum_{i=0}^{n}T_iT_i^*$ is isomorphic to $\mathbb{K}$, which is known to be the only closed non-trivial two-sided ideal. 
Then the quotient algebra $E_{n+1}/\K$ is isomorphic to $\cO_{n+1}$ with identification $S_i=\pi(T_i)$, where  $\pi$ is the quotient map. 

For a natural number $n$, we denote by $M_n$ the Moore space, the mapping cone of the map $n : S^1\ni z\mapsto z^n\in S^1$ :
$$M_n \colon=([0,1]\times S^1) \sqcup S^1/\sim,$$ where $(0, z)\sim (0,1)$ and $(1, z)\sim z^n$ for every $z\in S^1$. 
For cohomology and K-groups, we have 
$$H^0(M_n)=\mathbb{Z},\;H^1(M_n)=0, \;H^2(M_n)=\mathbb{Z}_n,\;H^k(M_n)=0 \textrm{ for } \;k\geq 2,$$ 
$$\tilde{K}^0(M_n)=\mathbb{Z}_n, \; K^1(M_n)=0,$$
(see \cite[Theorem 9.10]{H} for example).
Since $C_0(M_n, pt)$ and $\mathcal{O}_{n+1}$ have the same K-theory and they are in the bootstrap class, 
they are KK-equivalent (see \cite[Section 22.3]{Bl}). 

The mod $n$ K-group of the pointed space $(X, x_0)$ is originally defined by 
$$\tilde{K}^i(X ; \mathbb{Z}_n)\colon=\tilde{K}^i(X\wedge M_n), \; i=0, 1.$$
We refer to \cite{A1, A2} for the mod $n$ K-theory, and refer to \cite[Section 8]{RS} for an operator algebraic aspect of it. 
The Bott periodicity of the K-theory induces the Bott periodicity of the mod $n$ K-theory.  
By the KK-equivalence of $C_0(M_n, pt)$ and $\mathcal{O}_{n+1}$, the identification 
$$
\tilde{K}^i(X\wedge M_n)=K_i(C_0(X, x_0)\otimes C_0(M_n, pt))\cong K_i(C_0(X, x_0)\otimes \mathcal{O}_{n+1})
$$
is natural in the variable $X$ (see \cite[Theorem 6.4]{S}). 
We can identify the Bockstein exact sequence with 
the 6-term exact sequence 
$$\xymatrix{
\tilde{K}^0(X)\ar[r]^{-n}&\tilde{K}^0(X)\ar[r]^{\rho}&\tilde{K}^0(X ; \mathbb{Z}_n)\ar[d]^{\beta}\\
K^1(X ; \mathbb{Z}_n)\ar[u]^{\beta}&K^1(X)\ar[l]^{\rho}&K^1(X).\ar[l]^{-n}
}$$
arising from the exact sequence 
$$0\to C_0(X,x_0)\otimes \mathbb{K}\to C_0(X,x_0)\otimes E_{n+1}\to C_0(X,x_0)\otimes \mathcal{O}_{n+1}\to 0.$$
The map $\beta$ is called Bockstein map, and $\rho$ is called the reduction map.
We frequently identify $\beta$ with the index map or the exponential map in the $6$-term exact sequence.

\begin{lem}\label{wer}
We have the following isomorphisms from the Bockstein exact sequence $\colon$
\begin{align*}
\rho \colon \tilde{K}^0(M_n)\to \tilde{K}^0(M_n ; \mathbb{Z}_n),\quad \beta \colon \tilde{K}^1(M_n ; \mathbb{Z}_n)\to \tilde{K}^0(M_n).
\end{align*}
\end{lem}

The K-theory has a multiplication $\mu$ defined by the external tensor product of vector bundles:
$$\xymatrix{
\mu : K^0(X)\otimes K^0(Y)\ar[r]&K^0(X\times Y)\\
\tilde{K}^0(X)\otimes \tilde{K}^0(Y)\ar[u]\ar[r]&\tilde{K}^0(X\wedge Y)\ar[u]
}
$$
We denote the diagonal map by $\Delta_X : X\to X\times X$. This gives the ring structure of $K^0(X)$ :
$$x\cdot y :=\Delta_X^*\mu(x\otimes y), \;x, y\in K^0(X).$$
This induce the ring structure of $\tilde{K}^0(X)$ by $\Delta_X : X\to X\wedge X$ :
$$x\cdot y :=\Delta_X^*\mu(x\otimes y), \; x, y \in \tilde{K}^0(X).$$
From \cite[Chap.II, Theorem 5.9]{K}, the reduced K-group $\tilde{K}^0(X)$ is the set of nilpotent elements of $K^0(X)$, 
and in particular $\tilde{K}^0(\Sigma X)\cdot \tilde{K}^0(\Sigma X)=\{0\}$. 

The multiplication $\mu$ extends to $\tilde{K}^i(X), \; i=0, 1$ by 
$$\mu : \tilde{K}^0(S^i\wedge X)\otimes \tilde{K}^0(S^j\wedge Y)\to \tilde{K}^0(S^{i+j}\wedge X\wedge Y),$$
with the property
$$T_{X, Y}^*\mu(y\otimes x)=(-1)^{ij}\mu(x\otimes y), \; x\in \tilde{K}^i(X), \; y\in \tilde{K}^j(Y)$$
where the map $T_{X,Y} : X\wedge Y\to Y\wedge X$ is the exchange of the coordinates (see \cite[Chap. II section 5.30]{K}).
In a similar way, the multiplication $\mu$ defines the following :
$$\mu_L : \tilde{K}^i(X)\otimes \tilde{K}^j(Y ; \mathbb{Z}_n)\to \tilde{K}^{i+j}(X\wedge Y ; \mathbb{Z}_n),$$
$$\mu_R : \tilde{K}^i(X ; \mathbb{Z}_n)\otimes \tilde{K}^j(Y)\to \tilde{K}^{i+j}(X\wedge Y ; \mathbb{Z}_n),$$
with the same property (see \cite[Section 3]{A1}):
$$T_{X,Y}^*\mu_R(y\otimes x)=(-1)^{ij}\mu_L(x\otimes y), \; x\in \tilde{K}^i(X),\; y\in \tilde{K}^j(Y).$$
The multiplications $\mu$, $\mu_L$ and $\mu_R$ are compatible with the reduction $\rho$ and the map $\delta$ :
$$\mu_R(\rho \otimes{\rm id})=\rho\mu ,\;  \beta(\mu_R({\rm id}\otimes {\rm id}))=\mu(\beta\otimes {\rm id}),$$
$$\mu_L({\rm id} \otimes \rho)=\rho\mu ,\; \beta(\mu_L({\rm id}\otimes {\rm id}))=\mu({\rm id}\otimes \beta).$$
Since the identification $\tilde{K}^i(X;\mathbb{Z}_n)\cong K_i(C_0(X, x_0)\otimes\mathcal{O}_{n+1})$ is natural, it is compatible with the Kasparov product, and the multiplications $\mu_L$ and $\mu_R$ extend to 
\begin{align*}
\mu_L\colon &K_i(C(X))\otimes K_j(C(Y)\otimes\mathcal{O}_{n+1})\to K_{i+j}(C(X\times Y)\otimes\mathcal{O}_{n+1})\\
\mu_R\colon &K_i(C(X)\otimes \mathcal{O}_{n+1})\otimes K_j(C(Y))\to K_{i+j}(C(X\times Y)\otimes\mathcal{O}_{n+1}).
\end{align*}
In particular, for $u\in U((C(X)\otimes\mathcal{O}_{n+1}))$ and a projection $p\in C(X)\otimes \mathbb{M}_{m}$, we have
$$\mu_L([p]_0\otimes [u]_1)=[p\otimes u+(1_m-p)\otimes 1_{\mathcal{O}_{n+1}}]_1\in K_1(C(X\times X, \mathbb{M}_m\otimes\mathcal{O}_{n+1}))=K_1(C(X\times X, \mathcal{O}_{n+1})).$$

 We also use the K\"{u}nneth theorem of the reduced K-theory.
\begin{thm}[{\cite[Theorem 23.1.3]{Bl}}]\label{Kunn}
For pointed spaces $X$ and $Y$, we have the following exact sequence
$$0\to \bigoplus_{i=0,1}\tilde{K}^i(X)\otimes\tilde{K}^{i+*}(Y)\to \tilde{K}^*(X\wedge Y)\to \bigoplus_{i=0,1} 
\operatorname{Tor}(\tilde{K}^i(X), \tilde{K}^{i+1-*}(Y))\to 0,$$
that splits unnaturally.
\end{thm}

We note that the map $\tilde{K}^i(X)\otimes\tilde{K}^j(Y)\to \tilde{K}^{i+j}(X\wedge Y)$ above is given by the multiplication $\mu$.

Puppe sequence yields the following lemmas.

\begin{lem}[{\cite[Section 10, Proposition 3.4]{H}}]\label{yab}
For compact pointed spaces $X$ and $Y$, the sequence $X\vee Y\to X\times Y\to X\wedge Y$ induces 
a split exact sequence
$$0\to\tilde{K}^i(X\wedge Y)\to\tilde{K}^i(X\times Y)\to \tilde{K}^i(X)\oplus \tilde{K}^i(Y)\to 0.$$
The splitting is given by the projections ${\rm Pr}_X\colon X\times Y\to X$ and ${\rm Pr}_Y\colon X\times Y\to Y$.
\end{lem}

We have the diagram below
$$\xymatrix{
K^i(X\times Y)&K^i(X)\ar[l]^{\mu(\cdot \otimes 1)}\\
\tilde{K}^i(X\times Y)\ar[u]&\tilde{K}^i(X)\ar[u]\ar[l]^{{\rm Pr}^*_X}
}$$
where $1\in K^0(\{y_0\})$. 
So we identify the map ${\rm Pr}_X^*$ with the map $\mu(\cdot \otimes 1)$. We also identify ${\rm Pr}_Y^*$ with 
the map $\mu(1 \otimes \cdot)$ where $1\in K^0(\{x_0\})=\mathbb{Z}$.

\section{The group structure of $[X, \operatorname{Aut}\mathcal{O}_{n+1}]$}

\subsection{Description of the group structure}
Let $(X,x_0)$ be a pointed compact metrizable space. 
For every $\alpha\in\operatorname{Map}(X, \operatorname{Aut}\mathcal{O}_{n+1})$, we set
$$u_{\alpha}=\sum_{i=0}^{n}\alpha(1_{C(X)}\otimes S_i)(1_{C(X)}\otimes S^*_i)\in U(C(X)\otimes \mathcal{O}_{n+1}).$$ 
By \cite[Theorem 7.4]{D2}, the map 
$$[X, \operatorname{Aut}\mathcal{O}_{n+1}]\ni [\alpha]\mapsto [u_{\alpha}]_1\in K_1(C(X)\otimes \mathcal{O}_{n+1})=K^1(X;\Z_n)$$
is a bijection, though it is not a group homomorphism in general as we will see below. 
From the definition of $u_{\alpha}$, we have $u_{\alpha\beta}(x)=\alpha_x(u_\beta(x))u_\alpha(x)$, and 
$[u_{\alpha\beta}]_1=[u_\alpha]_1+[\alpha(u_\beta)]_1$. 
Thus to determine the group structure of $[X,\Aut\cO_{n+1}]$, it suffices to determine the map 
$$K_1(\alpha)\colon K_1(C(X)\otimes \cO_{n+1})\to K_1(C(X)\otimes \cO_{n+1}),$$
induced by $u(x)\mapsto \alpha_x(u(x))$.

\begin{thm}\label{Iz}
For every $\alpha \in \operatorname{Map}(X, \operatorname{Aut}\mathcal{O}_{n+1})$ and $a\in K_1(C(X)\otimes \mathcal{O}_{n+1})$, we have 
$$K_1(\alpha)(a)=a-[u_{\alpha}]_1\cdot \delta(a),$$
where $\delta : K_1(C(X)\otimes \mathcal{O}_{n+1})\to \operatorname{Tor}(\tilde{K}^0(X), \mathbb{Z}_n)$ is the index map.
\end{thm}
\begin{proof} 
For a given $b\in \Tor(K^0(X),\Z_n)$, we look for the preimage $\delta^{-1}(b)$ first. 
We may assume that $b$ is of the form $[p]_0-[1_m]_0$ with a projection $p\in C(X)\otimes\mathbb{M}_{2m}$ such that 
there exists a unitary $v\in C(X)\otimes\mathbb{M}_{2nm}$ satisfying $v(1_n\otimes q)v^{-1}=1_n\otimes p$, where 
$q=\Diag(1_m,0_m)$. 

Identifying $K^0(X)$ with $K_0(C(X)\otimes\K))$, we may replace $p$ and $q$ with $e\otimes p$ and $e\otimes q$ respectively, 
where $e=1_{E_{n+1}}-\sum_{i=0}^nT_iT_i^*$ is a minimal projection in $\K\subset E_{n+1}$. 
Furthermore, we may adjoin $1_{E_{n+1}}$ to $C(X)\otimes\K$, and 
$$c=[(1_{E_{n+1}}-e)\otimes q+e\otimes p]-[1_{E_{n+1}}\otimes q].$$
In what follows, we simply denote $1=1_{E_{n+1}}$, and often denote $1_{2m}$ for $1\otimes 1_{2m}$. 

We will construct a unitary $U\in C(X)\otimes E_{n+1}\otimes \mathbb{M}_{10m}$ satisfying 
$$U\Diag((1-e)\otimes q+e\otimes p,1_{4m},0_{4m})U^{-1}=\Diag(q,1_{4m},0_{4m}).$$
Expressing $v(x)=\sum_{i,j=1}^n e_{i,j}\otimes v_{i,j}(x)$, where $\{e_{i,j}\}_{1\leq i,j\leq n}$ is a system 
of matrix units $\M_n$, we let 
$$\tilde{v}(x)=(e+T_0T_0^*)\otimes 1_{2m}+\sum_{i,j=1}^nT_iT_j^*\otimes v_{i,j}(x).$$ 
Then $\tilde{v}$ is a unitary in $C(X)\otimes E_{n+1}\otimes \mathbb{M}_{2m}$ satisfying 
$$\tilde{v}((1-e)\otimes q+e\otimes p)\tilde{v}^*=T_0T_0^*\otimes q+(1-T_0T_0^*)\otimes p.$$
Thus if we put 
$$U_1=
\Diag(\tilde{v},
\left(
\begin{array}{cccc}
0 &0 &1_{2m} &0\\
0 &1_{2m} &0 &0  \\
1_{2m} &0 &0 &0  \\
0 &0 &0 &1_{2m} 
\end{array}
\right)
),$$ 
we get 
\begin{align*}
\lefteqn{U_1\Diag((1-e)\otimes q+e\otimes p,1_{4m},0_{4m})U_1^{-1}} \\
 &=\Diag(T_0T_0^*\otimes q+(1-T_0T_0^*)\otimes p,0_{2m},1_{4m},0_{2m}).
\end{align*}
Let 
$$U_2=\Diag(\left(
\begin{array}{cc}
T_0T_0^*\otimes 1_{2m} &(1-T_0T_0^*)\otimes 1_{2m}  \\
(1-T_0T_0^*)\otimes 1_{2m} & T_0T_0^*\otimes 1_{2m}
\end{array}
\right),1_{6m}
),$$
Then \begin{align*}
\lefteqn{U_2\Diag(T_0T_0^*\otimes q+(1-T_0T_0^*)\otimes p,0_{2m},1_{4m},0_{2m})U_2^{-1}} \\
 &=\Diag(T_0T_0^*\otimes q,(1-T_0T_0^*)\otimes p,1_{4m},0_{2m}).
\end{align*}
Let $U_3=\Diag(1_{2m},V_1,V_2)$ with 
$$V_1=\left(
\begin{array}{cc}
1_{2m}-T_0T_0^*\otimes p &T_0T_0^*\otimes p  \\
T_0T_0^*\otimes p &1_{2m}-T_0T_0^*\otimes p 
\end{array}
\right),$$
$$V_2=\left(
\begin{array}{cc}
T_0T_0^*\otimes q &1_{2m}-T_0T_0^*\otimes q  \\
1_{2m}-T_0T_0^*\otimes q &T_0T_0^*\otimes q 
\end{array}
\right).$$
Then 
\begin{align*}
\lefteqn{U_3\Diag(T_0T_0^*\otimes q,(1-T_0T_0^*)\otimes p,1_{4m},0_{2m})U_3^{-1}} \\
 &=\Diag(T_0T_0^*\otimes q,1\otimes p,1_{2m}-T_0T_0^*\otimes p,T_0T_0^*\otimes q,1_{2m}-T_0T_0^*\otimes q).
\end{align*}
Let 
$$U_4=\Diag(1_{2m},\left(
\begin{array}{ccc}
T_0\otimes 1_{2m} &0 &(1-T_0T_0^*)\otimes 1_{2m}  \\
0 &1_{2m} &0  \\
0 &0 &T_0^*\otimes 1_{2m} 
\end{array}
\right)
,1_{2m}).$$
Then 
\begin{align*}
\lefteqn{U_4\Diag(T_0T_0^*\otimes q,1\otimes p,1_{2m}-T_0T_0^*\otimes p,T_0T_0^*\otimes q,1_{2m}-T_0T_0^*\otimes q)U_4^{-1}} \\
 &=\Diag(T_0T_0^*\otimes q,T_0T_0^*\otimes p,1_{2m}-T_0T_0^*\otimes p,q,1_{2m}-T_0T_0^*\otimes q)
\end{align*}
Let 
$$U_5=\left(
\begin{array}{ccccc}
T_0T_0^*\otimes q &0 &0 &0 &1_{2m}-T_0T_0^*\otimes q  \\
0 &T_0T_0^*\otimes p &1_{2m}-T_0T_0^*\otimes p &0 &0  \\
0 &1_{2m}-T_0T_0^*\otimes p &T_0T_0^*\otimes p &0 &0  \\
0 &0 &0 &1_{2m} &0  \\
1_{2m}-T_0T_0^*\otimes q &0 &0 &0 &T_0T_0^*\otimes q 
\end{array}
\right)
.$$
Then 
$$U_5\Diag(T_0T_0^*\otimes q,T_0T_0^*\otimes p,1_{2m}-T_0T_0^*\otimes p,q,1_{2m}-T_0T_0^*\otimes q)U_5^{-1}
=\Diag(1_{2m},1_{2m},q,0_{4m}).$$
Let 
$$U_6=\left(
\begin{array}{ccccc}
0 &0 &0 &1_{2m} &0  \\
0 &1_{2m} &0 &0 &0  \\
1_{2m} &0 &0 &0 &0  \\
0 &0 &1_{2m} &0 &0  \\
0 &0 &0 &0 &1_{2m} 
\end{array}
\right).
$$
Then 
$$U_6\Diag(1_{2m},1_{2m},q,0_{4m})U_6^{-1}=\Diag(q,1_{4m},0_{4m}).$$
Thus if we put $U=U_6U_5U_4U_3U_2U_1$, we get 
$$U\Diag((1-e)\otimes q+e\otimes p,1_{4m},0_{4m})U^{-1}=\Diag(q,1_{4m},0_{4m}).$$

Recall that $\pi:E_{n+1}\to \cO_{n+1}$ is the quotient map, 
Since 
$$(\pi\otimes \id_{M_{10m}})(\Diag((1-e)\otimes q+e\otimes p,1_{4m},0_{4m}))=\Diag(q,1_{4m},0_{4m}),$$
the unitary $(\pi\otimes \id_{M_{10m}})(U)$ commutes with $\Diag(q,1_{4m},0_{4m})$. 
Let
$$W=\Diag(q,1_{4m},0_{4m})(\pi\otimes \id_{M_{10m}})(U^{-1})\Diag(q,1_{4m},0_{4m}),$$
which we regard as a unitary in $C(X,\cO_{n+1}\otimes M_{5m})$. 
Then by the definition of the index map, we get $\delta([W]_1)=b$. 

Let 
$$V(x)=\sum_{i,j=1}^nS_iS_j^*\otimes v_{i,j}(x).$$
Direct computation yields
$$W=\left(
\begin{array}{ccc}
0 &V^*(S_0^*\otimes p) &S_0S_0^*\otimes q  \\
S_0\otimes q &0 &1_{2m}-S_0S_0^*\otimes q  \\
0 &1_{2m}-S_0S_0^*\otimes p+S_0S_0^*S_0^*\otimes p &0 
\end{array}
\right).
$$
Let $\beta=\Ad u_\alpha^*\circ \alpha$. 
Then $K_1(\alpha)=K_1(\beta)$, and $\beta(S_i)=S_iu_\alpha$. 
Now 
\begin{align*}
\lefteqn{W^*(\beta\otimes \id_{M_{5m}})(W)} \\
 &=
\left(
\begin{array}{ccc}
0 &S_0^*\otimes q & 0 \\
(S_0\otimes p)V &0 &1_{2m}-S_0S_0^*\otimes p+S_0^2S_0^*\otimes p  \\
S_0S_0^*\otimes q &1_{2m}-S_0S_0^*\otimes q &0 
\end{array}
\right) \\
&\times  
 \left(
\begin{array}{ccc}
0 &V^*(u_{\alpha}^{-1}S_0^*\otimes p) &S_0S_0^*\otimes q  \\
S_0u_\alpha\otimes q &0 &1_{2m}-S_0S_0^*\otimes q  \\
0 &1_{2m}-S_0S_0^*\otimes p+S_0 S_0^*u_\alpha^{-1}S_0^*\otimes p &0 
\end{array}
\right)
\\
 &=\left(
\begin{array}{ccc}
u_\alpha\otimes q &0 &0  \\
 0&S_0u_\alpha^{-1}S_0^*\otimes p+1_{2m}-S_0S_0^*\otimes p &0  \\
 0&0 &1_{2m} 
\end{array}
\right)
 \\
 &=\Diag(u_\alpha\otimes q, 
 \left(
\begin{array}{cc}
S_0\otimes 1_{2m} &(1-S_0S_0^*)\otimes 1_{2m}  \\
0 &S_0^*\otimes 1_{2m} 
\end{array}
\right)
\left(
\begin{array}{cc}
u_\alpha^{-1}\otimes p+1\otimes (1_{2m}-p) &0  \\
0 &1_{2m} 
\end{array}
\right)\\
&\times \left(
\begin{array}{cc}
S_0^*\otimes 1_{2m} &0  \\
(1-S_0S_0^*)\otimes 1_{2m} &S_0\otimes 1_{2m} 
\end{array}
\right)
),
\end{align*}
whose $K_1$-class is 
$$[u_\alpha]_1([q]_0-[p]_0)=-[u_\alpha]_1\cdot b=-[u_\alpha]_1\cdot\delta([W]_1).$$
Thus
$$K_1(\alpha)([W]_1)=[W]_1-[u_\alpha]_1\cdot\delta([W]_1).$$
Since $K_1(\alpha)(a)=a$ and $\delta(a)=0$ hold for any $a\in \rho(K^1(X))$, we get 
$$K_1(\alpha)([W]_1+a)=[W]_1+a-[u_\alpha]_1\cdot\delta([W]_1+a),$$
which finishes the proof.  
\end{proof}

Recall that we identify the index map $\delta$ with the Bockstein map $\beta$. 
By Theorem \ref{Iz}, the group $[X, \operatorname{Aut}\mathcal{O}_{n+1}]$ is isomorphic to $(K^1(X ; \mathbb{Z}_n), \circ)$ with  
$$a\circ b\colon =a+b-a\cdot \beta(b), \;a, b\in K^1(X ;\mathbb{Z}_n).$$
Note that $(K^1(X ; \mathbb{Z}_n), \circ)$ is a group extension 
$$0\to K^1(X)\otimes \Z_n\to (K^1(X ; \mathbb{Z}_n), \circ)\xrightarrow{\hat{\beta}} (1+\operatorname{Tor}(\tilde{K}^0(X),\Z_n))^\times\to 0,$$
where $\hat{\beta}(a)=1-\beta(a)$. 
We denote the inverse of an element $a\in (\tilde{K}^1(X ; \mathbb{Z}_n), \circ)$ by $a^{\circ (-1)}$.

\begin{lem}\label{inverse}
For any $a,b \in(\tilde{K}^1(X ; \mathbb{Z}_n), \circ)$, we have 
\begin{itemize}
\item[$(1)$] $a^{\circ(-1)}=-a\cdot (1-\beta(a))^{-1}$.
\item[$(2)$] $a^{\circ(-1)}\circ b\circ a=b+a\cdot\beta(b)-\beta(a)\cdot b$. 
In particular, if $b\in K^1(X)\otimes \Z_n=\ker \beta$, we have 
$a^{\circ(-1)}\circ b\circ a=(1-\beta(a))b$. 
\end{itemize}
\end{lem}

\begin{proof}
Direct computation yields
\begin{align*}
(-a\cdot (1-\beta(a))^{-1}\circ a=&-a\cdot (1-\beta(a))^{-1}+a+a\cdot (1-\beta(a))^{-1}\cdot \beta(a)\\
=&a+a\cdot(1-\beta(a))^{-1}\cdot (\beta(a)-1)\\
=&0,
\end{align*}
showing the first equation. 
The second one follows from the first one. 
\end{proof} 

Now we discuss the relationship between the two groups $[X,\Aut E_{n+1}]$ and $[X,\Aut \cO_{n+1}]$. 
Let $H_1$ be the set of vectors of norm $1$ in a separable infinite dimensional Hilbert space $H$. 
Then $H_1$ is contractible. 
Indeed, we can identify $H_1$ with the set $\{f\in L^2[0,1]\mid ||f||_{2}=1\}$,  
and define a homotopy $h_t : H_1 \to H_1$ sending $f$ to $(1_{[0, t]}f+1_{[t, 1]})/||1_{[0, t]}f+1_{[t, 1]}||_2$ 
where $1_{[a,b]}$ is the characteristic function of $[a,b]$. 
This gives a deformation retraction of $H_1$ to the set $\{1_{[0, 1]}\}$, and the space $H_1$ is contractible (see \cite{RW}). 
Since the group $S^1=\{z\in \C;\;|z|=1\}$ freely acts on $H_1$ by multiplication, we can adopt $H_1$ as a model of the universal principal 
$S^1$-bundle $\operatorname{E}S^1$ and identify the classifying space $\operatorname{B}S^1$ of $S^1$ with the set of all 
minimal projections. 
The space $\operatorname{B}S^1$ is the Eilenberg-Maclane space $K(\mathbb{Z}, 2)$ and we identifies the homotopy set $[X, \operatorname{B}S^1]$ 
with $H^2(X)$ via the Chern classes of the line bundles.

Let $\eta$ be the map $\operatorname{Aut}E_{n+1}\ni \alpha\mapsto\alpha(e)\in \operatorname{B}S^1$.
We denote by $\eta_*$ the induced map $\eta_*:[X, \operatorname{Aut}E_{n+1}]\to H^2(X)$, 
which is a group homomorphism with image in $\Tor (H^2(X),\Z_n)$ (see \cite[Theorem 3.15]{ST}). 
We will show that the two maps $\eta_*$ and $\hat{\beta}$ are compatible.
Let $\mathcal{P}(\mathbb{K})$ be the set of all projections of $\mathbb{K}$.
We remark that the map $[X, \mathcal{P}(\mathbb{K})]\ni [p]\mapsto [p]_0\in K^0(X)$ is well-defined by the definition of the $K_0$-group.
Since $\cO_{n+1}$ is the quotient of $E_{n+1}$ by a unique non-trivial closed two sided ideal, every element in $\Aut E_{n+1}$ 
induces an element in $\Aut \cO_{n+1}$, which gives a group homomorphism from $\Aut E_{n+1}$ to $\Aut \cO_{n+1}$. 
We denote by $q$ the group homomorphism from $[X, \operatorname{Aut}E_{n+1}]$ to $[X,\Aut \cO_{n+1}]$ induced by 
this homomorphism. 

\begin{prop}\label{commutes}
Let $q$ be as above, and let $l\colon H^2(X)\to K^0(X)$ be a map induced by the map 
$\operatorname{B}S^1\to \mathcal{P}(\mathbb{K})$ 
where we identify $\operatorname{B}S^1$ with the set of all minimal projections.
Then we have the following commutative diagram
$$
\begin{CD} [X, \operatorname{Aut}E_{n+1}] @>\eta_*>> \Tor(H^2(X),\Z_n) \\
@VV q V @VV l V\\
[X,\Aut \cO_{n+1}]@ > \hat{\beta} >> (1+\operatorname{Tor}(\tilde{K}^0(X),\Z_n))^\times\\
\end{CD}
$$
\end{prop}

\begin{proof} 
For $\alpha\in \operatorname{Map}(X,\Aut E_{n+1})$, we denote by $\tilde{\alpha}$ the map in 
$\operatorname{Map}(X,\Aut \cO_{n+1})$ induced by $\alpha$. 
Then with the identification of $[X,\Aut \cO_{n+1}]$ and $K^1(X;\Z_n)$, the map $q$ sends $[\alpha]$ 
to $[u_{\tilde{\alpha}}]_1$. 

One has $l\circ\eta_*([\alpha])=[\alpha(1_{C(X)}\otimes e)]_0\in K_0(C(X))$ for every $\alpha\in\operatorname{Map}(X, \operatorname{Aut}E_{n+1})$ by definition.
Since $\beta$ is given by the index map $\delta \colon K_1(C(X)\otimes \mathcal{O}_{n+1})\to K_0(C(X))$.
We compute the index ${\rm ind}\, [u_{\tilde{\alpha}}]_1$. 
We have a unitary lift $V\in U(\mathbb{M}_2(C(X)\otimes \mathcal{O}_{n+1}))$ of the unitary $u_{\tilde{\alpha}}\oplus u^*_{\tilde{\alpha}}$ :
\begin{align*}
V=\left(
\begin{array}{cc}
\sum_{i=1}^{n+1}\alpha(1\otimes T_i)T_i^*&\alpha(1\otimes e)\\
1\otimes e&\sum_{i=1}^{n+1}1\otimes T_i\alpha(1\otimes T_i^*)
\end{array}\right).
\end{align*}
Direct computation yields
$$V(1\oplus 0)V^*=(1-\alpha(1\otimes e))\oplus (1\otimes e)$$
where we write $1_{C(X)}\otimes e$ simply by $1\otimes e$.
Hence we have $${\rm ind}[u_{\tilde{\alpha}}]_1=[1-\alpha(1\otimes e)]_0+[1\otimes e]_0-[1]_0=1-[\alpha(1\otimes e)]_0\in K_0(C(X)\otimes\mathbb{K}).$$
Now we have $1-{\rm ind}[u_{\tilde{\alpha}}]_1=[\alpha(1\otimes e)]_0$, and this proves the statement.
\end{proof}

\begin{lem}\label{com}
We have the following commutative diagram with exact rows
$$\xymatrix@!C{
K^1(X)\ar[r]\ar@{=}[d]&[X, \operatorname{Aut}E_{n+1}]\ar[d]^{q}\ar[r]^{\eta_*}&\Tor(H^2(X),\Z_n)\ar[d]^{l}\\
K^1(X)\ar[r]^{\rho}&[X,\Aut \cO_{n+1}]\ar[r]^{\hat{\beta}}&(1+\Tor(\tilde{K}^0(X),\Z_n))^\times.
}$$
\end{lem}

\begin{proof} Let $\End E_{n+1}$ be the set of unital endomorphisms of $E_{n+1}$, and let $\End_0 E_{n+1}$ be its connected 
component of $\id$.  
Then the inclusion $\Aut E_{n+1}\subset \End_0E_{n+1}$ is a weak homotopy equivalence (see \cite[Theorem 3.14]{ST}). 
For $u\in U(E_{n+1})$, we denote by $\rho_u$ the unital endomorphism of $E_{n+1}$ defined by 
$\rho_u(T_i)=uT_i$. 
Then the correspondence $[u]_1\to [\rho_u]$ gives the map from $K^1(X)$ to $[X,\Aut E_{n+1}]$. 
The exactness follows from \cite[Theorem 3.15]{ST} and the Bockstein exact sequence.
The right  square commutes by Proposition \ref{commutes}.
The left square commutes because the following diagram commutes
$$\xymatrix{
u\in U(C(X)\otimes E_{n+1})\ar@{=}[d]\ar[r]&\operatorname{Map}(X, \operatorname{End}_0E_{n+1})\ni\alpha=\rho_{u}\ar[d]\\
u\in U(C(X)\otimes E_{n+1})\ar[r]^{\pi}&U(C(X)\otimes \mathcal{O}_{n+1})\ni u_{\tilde{\alpha}}=\pi(u)
}$$
where $\rho_{u}\colon X\ni x\mapsto \rho_{u_x}\in \operatorname{End}_0E_{n+1}$ for every $u\in U(C(X)\otimes E_{n+1})$.
\end{proof}

\subsection{An example of non-commutative $[X,\Aut \cO_{n+1}]$}
We first examine the ring structure of $K^*(M_n\times \Sigma M_n)$ to show that  
$[M_n\times \Sigma M_n, \operatorname{Aut}\cO_{n+1}]$ is a non-commutative group. 
By Lemma \ref{yab} and Theorem \ref{Kunn}, we have 
\begin{align*}
\tilde{K}^1(M_n\times \Sigma M_n)\cong &1\otimes \tilde{K}^1(\Sigma M_n)\oplus \tilde{K}^0(M_n)\otimes \tilde{K}^1(\Sigma M_n),\\
\tilde{K}^0(M_n\times \Sigma M_n)\cong &\tilde{K}^0(M_n)\otimes 1\oplus \tilde{K}^0(M_n\wedge \Sigma M_n).
\end{align*}
Therefore Lemma \ref{wer} yields $\tilde{K}^i(M_n\times \Sigma M_n)\cong \mathbb{Z}_n^{\oplus 2}$.
In particular, the map $\rho\colon \tilde{K}^i(M_n\times \Sigma M_n)\to\tilde{K}^i(M_n\times \Sigma M_n; \mathbb{Z}_n)$ is injective by the Bockstein exact sequence.

We determine a generator of $K^1(M_n ;\mathbb{Z}_n)\cong \tilde{K}^0(M_n)\cong \mathbb{Z}_n$.
Recall that the canonical gauge action $\lambda_z\colon S^1\to \operatorname{Aut}E_{n+1}$ is a generator of $\pi_1(\operatorname{Aut}E_{n+1})=\mathbb{Z}_n$ (see \cite[Theorem 2.36, 3.14 ]{ST}).
Therefore we have a homotopy
$$h\colon [0,1]\times S^1\to\operatorname{Aut}E_{n+1}$$
with $h_0(z)={\rm id}_{E_{n+1}}$, $h_1(z)=\lambda_z^n$, 
which extend $\lambda$ to a map
$$\lambda\colon M_n\to \operatorname{Aut}E_{n+1}$$
satisfying $\lambda\circ i=\lambda_z$ for the map $i\colon S^1\hookrightarrow M_n$. 
For the gauge action $\tilde{\lambda}$ of $\cO_{n+1}$, we get an extension $\tilde{\lambda}:M_n\to \Aut \cO_{n+1}$ 
in the same way. 

\begin{lem}\label{as}
We have the following isomorphisms :
$$i^* : [M_n, \operatorname{Aut}E_{n+1}] \ni [\lambda]\mapsto [\lambda_z]\in [S^1, \operatorname{Aut}E_{n+1}],$$
$$i^* : [M_n, \operatorname{Aut}\mathcal{O}_{n+1}]\ni [\tilde{\lambda}]\mapsto [\tilde{\lambda}_z]\in [S^1, \operatorname{Aut}\mathcal{O}_{n+1}].$$
\end{lem}
\begin{proof}
First, we show that $i^* : [M_n, \operatorname{Aut}\mathcal{O}_{n+1}]\to[S^1, \operatorname{Aut}\mathcal{O}_{n+1}]$ is an isomorphism.
By \cite{D2}, Puppe sequence $S^1\xrightarrow{n} S^1\xrightarrow{i} M_n\to S^2\to\dotsm $ gives  an exact sequence
$$\mathbb{Z}_n=\pi_1(\operatorname{Aut}\mathcal{O}_{n+1})\xleftarrow{n}\pi_1(\operatorname{Aut}\mathcal{O}_{n+1})\xleftarrow{i_*}[M_n, \operatorname{Aut}\mathcal{O}_{n+1}]\leftarrow 0.$$
Hence the map $i^*$ is an isomorphism of groups.

Similarly, the map $i_*\colon [M_n, \operatorname{Aut}E_{n+1}]\to [S^1, \operatorname{Aut}E_{n+1}]$ is an isomorphism by \cite[Theorem 2.36, 3.14]{ST}.
\end{proof}
\begin{lem}\label{m0}
For every $\alpha \in \operatorname{Map}(M_n, \operatorname{Aut}\mathcal{O}_{n+1})$, we have 
$K_1(\alpha)={\rm id}_{K^1(M_n ; \mathbb{Z}_n)}.$
In particular, we have $\tilde{K}^0(M_n)\cdot K^1(M_n ; \mathbb{Z}_n)=\tilde{K}^0(M_n)\cdot \tilde{K}^0(M_n)=\{0\}$.
\end{lem}
\begin{proof}
By Lemma \ref{as}
We have the following commutative diagram 
$$\xymatrix{
[M_n, \operatorname{Aut}\mathcal{O}_{n+1}]\ar@{=}[r]^{i^*}\ar[d]&[S^1, \operatorname{Aut}\mathcal{O}_{n+1}]\ar[d]\\
(K_1(C_0(M_n, pt)\otimes\mathcal{O}_{n+1}), \;\circ)\ar[r]^{K_1(r)}&(K_1(C_0(S^1, pt)\otimes \mathcal{O}_{n+1}), \;\circ),
}$$
where $r : C_0(M_n, pt)\to C_0(S^1, pt)$ is a restriction by $i : S^1\hookrightarrow M_n$. Since two vertical maps are group isomorphisms, the map $K_1(r)$ is a group homomorphism with respect to the multiplication $\circ$. We have $K_1(C(S^1)\otimes \mathcal{O}_{n+1})\xrightarrow{\delta}\tilde{K}^0(S^1)=0$, and it follows that $(K^1(S^1 ; \mathbb{Z}_n), \;+)=(K^1(S^1 ;\mathbb{Z}_n),\;\circ)$ by Theorem \ref{Iz}. Therefore two multiplications $\circ$ and $+$ coincide in $K^1(M_n ; \mathbb{Z}_n)$,
and we have $K_1(\alpha)={\rm id}_{K^1(M_n ;\mathbb{Z}_n)}$ and $K^1(M_n ;\mathbb{Z}_n)\cdot \tilde{K}^0(M_n)=0$.
Since the map $\beta$ is compatible with multiplication, we have $\tilde{K}^0(M_n)\cdot\tilde{K}^0(M_n)=\beta(\tilde{K}^1(M_n ;\mathbb{Z}_n)\cdot\tilde{K}^0(M_n))=\{0\}$.
\end{proof}

We denote by $a_{\lambda}$ the generator $[\tilde{\lambda}]\in [M_n, \operatorname{Aut}\mathcal{O}_{n+1}]=K^1(M_n ;\mathbb{Z}_n)$, and denote $g\colon=\beta(a_{\lambda})$.

By Lemma \ref{wer}, two elements $g$ and $\rho(g)$ are the generators of $\tilde{K}^0(M_n)$ and $\tilde{K}^0(M_n ;\mathbb{Z}_n)$ respectively.
By Lemma \ref{m0}, we have 
\begin{align*}
g\cdot g&=0 \in\tilde{K}^0(M_n),\\ 
a_{\lambda}\cdot g=&0 \in\tilde{K}^1(M_n ; \mathbb{Z}_n).
\end{align*}

Now, we determine the group $[M_n\times \Sigma M_n, \operatorname{Aut}E_{n+1}]$.
Since the reduction $\rho\colon \tilde{K}^1(M_n\times \Sigma M_n)\to \tilde{K}^1(M_n\times \Sigma M_n ; \mathbb{Z}_n)$ 
is injective, we regard $\tilde{K}^1(M_n\times \Sigma M_n)$ as a subgroup of $(\tilde{K}^1(M_n\times \Sigma M_n ; \mathbb{Z}_n), \circ)$.
From Lemma \ref{com}, we can regard $\tilde{K}^1(M_n\times \Sigma M_n)$ as a normal subgroup of the group 
$[M_n\times \Sigma M_n, \operatorname{Aut}E_{n+1}]$ too. 
Consider the map 
$$\Lambda \colon=\lambda\circ{\rm Pr}_{M_n}\colon M_n\times \Sigma M_n\to \operatorname{Aut}E_{n+1}.$$
By definition, we have $q([\Lambda])=[u_{\tilde{\Lambda}}]_1={\rm Pr}^*_{M_n}([u_{\tilde{\lambda}}]_1)
=\mu_R(a_{\lambda}\otimes 1)\in\tilde{K}^1(M_n\times \Sigma M_n ;\mathbb{Z}_n).$

\begin{prop}\label{poop}
The group homomorphism $q\colon [M_n\times \Sigma M_n, \operatorname{Aut}E_{n+1}]\to [M_n\times\Sigma M_n, \operatorname{Aut}\mathcal{O}_{n+1}]$ is injective.
\end{prop}

\begin{proof} Note that the K\"{u}nneth formula implies $H^2(M_n\times \Sigma M_n)\cong \Z_n$. 
Since $\hat{\beta}(\mu_R(a_\lambda\otimes 1))=1-\mu(g\otimes 1)$ has order $n$, 
and $\hat{\beta}(\mu_R(a_\lambda\otimes 1))=l(\eta_*([\Lambda]))$, the element $\eta_*([\Lambda])$ is 
a generator of $H^2(M_n\times \Sigma M_n)$, and $l$ is injective. 
Thus the statement follows from Lemma \ref{com}.
\end{proof}

\begin{thm}\label{ex1} With the above notation, the group 
$[M_n\times \Sigma M_n, \operatorname{Aut}E_{n+1}]$ is isomorphic to the Heisenberg group 
$$ \mathbb{Z}_n^{\oplus 2}
\rtimes_{\left(
\begin{array}{cc}
1 &1  \\
0 &1
\end{array}
\right)
} \mathbb{Z}_n.$$ 
\end{thm}

\begin{proof} We already know that the group $[M_n\times \Sigma M_n, \operatorname{Aut}E_{n+1}]$ isomorphic to 
the subgroup of $(K^1(M_n\times \Sigma M_n;\Z_n),\circ)$ generated by $K^1(M_n\times \Sigma M_n)$ and 
$[u_{\tilde{\Lambda}}]_1$.   
Since the order of $[u_{\tilde{\Lambda}}]_1$ is $n$, the group is a semi-direct product $(\Z_n\times \Z_n)\rtimes \Z_n$. 
To determine the group structure, it suffices to compute the action of 
$\hat{\beta}([u_{\tilde{\Lambda}}]_1)=1-\mu(g\otimes 1)$ 
on $K^1(M_n\times \Sigma M_n)$ by multiplication. 
Since $\tilde{K}^1(M_n\times \Sigma M_n)=\langle\mu(1\otimes u)\rangle\oplus \langle\mu(g\otimes u)\rangle\cong\mathbb{Z}_n\oplus\mathbb{Z}_n$, 
and $g\cdot g=0$, we get the statement. 
\end{proof}

\begin{cor} The groups $[M_n\times \Sigma M_n, \operatorname{Aut}E_{n+1}]$ and 
$[M_n\times \Sigma M_n, \operatorname{Aut}\mathcal{O}_{n+1}]$ are non-commutative for any $n\geq 2$. 
In particular, two spaces $\operatorname{BAut}\mathcal{O}_{n+1}$ and $\operatorname{BAut}E_{n+1}$ are not H-spaces.
\end{cor}

\begin{rem}
If $n$ is an odd number, we can actually show 
 $$[M_n\times\Sigma M_n, \operatorname{Aut}\mathcal{O}_{n+1}]\cong [M_n\times \Sigma M_n,\Aut E_{n+1}]\times \mathbb{Z}_n.$$
\end{rem}

\section{Continuous fields of Cuntz algebras}
We first review Dadarlat's results on the continuous fields of the Cuntz algebras.
We refer to \cite[Definition 10.1.2, 10.1.3]{Dix} for the definition of the continuous fields of  C*-algebras.  
A locally trivial continuous field of  a C*-algebra $A$ is the section algebra of a locally trivial fiber bundle with the fibre $A$, which is an associated bundle of a principal $\operatorname{Aut}A$ bundle. 
By \cite[Theorem 1.1]{D2}, all continuous fields of $\mathcal{O}_{n+1}$ over finite CW-complexes are locally trivial.
So we identify the continuous fields of $\mathcal{O}_{n+1}$ over finite CW-complexes with principal $\operatorname{Aut}\mathcal{O}_{n+1}$ bundles.

For a compact Hausdorff space $X$, we denote by ${\rm Vect}_m\;(X)$ the set of the vector bundles of rank $m$. 
Dadarlat investigated continuous fields of $\mathcal{O}_{n+1}$ over $X$ arising 
from $E\in {\rm Vect}_{n+1}(X)$, which are Cuntz-Pimsner algebras. 
We refer to \cite{KT} and \cite{Pim} for Cuntz-Pimsner algebras. 
Fixing a Hermitian structure of $E$, we get a Hilbert $C(X)$-module from $E$, which we regard as a $C(X)$-$C(X)$-bimodule.  
Then the Pimsner construction gives the Cuntz-Pimsner algebra $\cO_E$, which is the quotient of 
$\mathcal{T}_E$ by $\mathcal{K}_E$. 
The algebra $\mathcal{O}_E$ is a continuous field of $\cO_{n+1}$ over $X$. 
We denote by $\theta_E : C(X)\to \mathcal{O}_E$ the natural unital inclusion.

\begin{thm}[{\cite[Theorem 4.8]{Pim}}]\label{op}
Let $X$ be a compact  Hausdorff space, and let $E$ be a vector bundle over $X$. Then we have the following exact sequence
$$\xymatrix{
K_0(C(X))\ar[r]^{1-[E]}&K_0(C(X))\ar[r]^{\theta_{E}}&K_0(\mathcal{O}_E)\ar[d]\\
K_1(\mathcal{O}_E)\ar[u]&K_1(C(X))\ar[l]^{\theta_E}&K_1(C(X))\ar[l]^{1-[E]}
}$$
where the map $\theta_E : C(X)\to \mathcal{O}_E$ is the natural inclusion, and the map $1-[E]$ is the multiplication by $1-[E]\in K^0(X)$.
\end{thm}

Dadarlat found an invariant to classify the $C(X)$-linear isomorphism classes of $\mathcal{O}_E$.
\begin{thm}[{\cite[Theorem 1.1]{D1}}]
Let $X$ be a compact metrizable space, and let $E$ and $F$ be vector bundles of rank $\geq 2$ over $X$. Then there is a unital $*$-homomorphism $\varphi : \mathcal{O}_E\to \mathcal{O}_F$ with $\varphi\circ\theta_E=\theta_F$ if and only if $(1-[E])\cdot K^0(X)\subset(1-[F])\cdot K^0(X)$. Moreover we can take $\varphi$ to be an isomorphism if and only if $(1-[E])\cdot K^0(X)=(1-[F])\cdot K^0(X)$.
\end{thm}

The key observation of Dadarlat is that if there is a $C(X)$-linear isomorphism $\varphi : \mathcal{O}_E\to\mathcal{O}_F$, 
we have $(1-[E])\cdot K^0(X)=\operatorname{Ker}K_0(\theta_E)=\operatorname{Ker}K_0(\theta_F)=(1-[F])\cdot K^0(X)$ 
by the exact sequence of Theorem \ref{op}. 
Dadarlat also estimate the cardinality of the set of the $C(X)$-linear isomorphism classes of $\mathcal{O}_E$.
We denote $\lceil x\rceil\colon ={\rm min}\{k\in\mathbb{Z} \colon k\geq x\}$.
\begin{thm}\label{DG}
Let $X$ be a finite connected CW-complex with $\operatorname{Tor}(H^*(X), \mathbb{Z}_n)=0$.
Then the following holds.\\
${\rm (1)}$ $|\tilde{K}^0(X)\otimes \mathbb{Z}_n|=|\tilde{H}^{even}(X, \mathbb{Z}_n)|$.\\
$\rm (2)$ If  $n\geq\lceil({\rm dim}\;X-3)/3\rceil$, the set $\{[\mathcal{O}_E];\; E\in {\rm Vect}_{n+1}(X)\}$ exhausts 
all the isomorphism classes of continuous fields of $\cO_{n+1}$ over $X$, and its cardinality is  
$|\tilde{K}^0(X)\otimes \mathbb{Z}_n|$. 
\end{thm}

Our goal in this section is to remove the restriction $n\geq\lceil({\rm dim}\;X-3)/3\rceil$ from the above statement 
using a localization trick. 
In fact, all the necessary algebraic arguments for the proof are already in Dadarlat's paper \cite{D1} 

Let $P_n$ be the set of all prime numbers $p$ with $(n,p)=1$, 
and let $\M_{(n)}$ be the UHF algebra 
$$\M_{(n)}=\bigotimes_{p\in P_n}\M_{p^\infty}.$$
This is the unique UHF algebra satisfying $K_0(\M_{(n)})=\mathbb{Z}_{(n)}$ where $\mathbb{Z}_{(n)}$ is a localization of $\mathbb{Z}$ by $n$. 
Assume that $r$ is a natural number with $(n,r)=1$. 
Then the K-groups of $\mathcal{O}_{nr+1}\otimes \mathbb{M}_{(n)}$ are 
 $$K_0(\mathcal{O}_{nr+1}\otimes \mathbb{M}_{(n)})=\mathbb{Z}_{nr}\otimes\mathbb{Z}_{(n)}=\mathbb{Z}_n=\langle[1]_0\rangle, \;\;
 K_1(\mathcal{O}_{nr+1}\otimes \mathbb{M}_{(n)})=0.$$ 
Therefore Kirchberg and Phillips' classification theorem \cite[Theorem 4.2.4]{Phill} yields $\mathcal{O}_{nr+1}\otimes \mathbb{M}_{(n)}\cong\mathcal{O}_{n+1}$. Let $F_r$ be a vector bundle over $X$ of rank $nr+1$. Then we have a continuous field of $\mathcal{O}_{n+1}$ of the form $\mathcal{O}_{F_r}\otimes \mathbb{M}_{(n)}$.

\begin{dfn}
We denote by $\mathcal{O}(X)_n$ the $C(X)$-linear isomorphism classes of continuous fields of the Cuntz algebra 
$\mathcal{O}_{n+1}$ over $X$ of the form $\mathcal{O}_{F_r}\otimes \mathbb{M}_{(n)}$ for $F_r\in{\rm Vect}_{nr+1}(X)$ with $(n, r)=1$.
\end{dfn}

Note that we have $K_*(C(X)\otimes \mathbb{M}_{(n)})=K^*(X)\otimes \mathbb{Z}_{(n)}$. 
Following Dadarlat' s argument, we consider an ideal $(1-[F_r])K^0(X)\otimes \mathbb{Z}_{(n)}$ of the ring $K^0(X)\otimes \mathbb{Z}_{(n)}$. 
\begin{lem}\label{inv}
Let $X$ be a finite connected CW-complex.
Let $F_r$ and $F_R$ be vector bundles over $X$ of rank $nr+1$ and $nR+1$ respectively, with $(n, r)=(n, R)=1$. If $\mathcal{O}_{F_r}\otimes\mathbb{M}_{(n)}$ is $C(X)$-linearly isomorphic to $\mathcal{O}_{F_R}\otimes \mathbb{M}_{(n)}$, we have $(1-[F_r]) K^0(X)\otimes\mathbb{Z}_{(n)}=(1-[F_R])K^0(X)\otimes \mathbb{Z}_{(n)}$. 
\end{lem}
\begin{proof}
Let $\varphi \colon \mathcal{O}_{F_r}\otimes\mathbb{M}_{(n)} \to\mathcal{O}_{F_R}\otimes \mathbb{M}_{(n)}$ be a $C(X)$-linear isomorphism.
First, we show that the following diagram induces a commutative diagram of $K_0$-groups :
$$\xymatrix{
C(X)\otimes \mathbb{M}_{(n)}\ar[rr]^{\theta_{F_r}\otimes {\rm id}}\ar@{=}[d]&&\mathcal{O}_{F_r}\otimes \mathbb{M}_{(n)}\ar[d]^{{\rm id}\otimes 1\otimes {\rm id}}\\
C(X)\otimes \mathbb{M}_{(n)}\ar[rr]^{\theta_{F_r}\otimes {\rm id}\otimes 1}\ar@{=}[d]&&\mathcal{O}_{F_r}\otimes \mathbb{M}_{(n)}\otimes \mathbb{M}_{(n)}\ar[d]^{\varphi\otimes {\rm id}}\\
C(X)\otimes \mathbb{M}_{(n)}\ar[rr]^{\theta_{F_R}\otimes {\rm id}\otimes 1}\ar@{=}[d]&&\mathcal{O}_{F_R}\otimes \mathbb{M}_{(n)}\otimes\mathbb{M}_{(n)}\\
C(X)\otimes \mathbb{M}_{(n)}\ar[rr]^{\theta_{F_R}\otimes {\rm id}}&&\mathcal{O}_{F_R}\otimes \mathbb{M}_{(n)}.\ar[u]^{{\rm id}\otimes 1\otimes{\rm id}}
}$$
The middle square of the diagram commutes because $\varphi$ is $C(X)$-linear.
By \cite[Theorem 2.2]{DW}, two $*$-homomorphisms $1\otimes {\rm id}, {\rm id}\otimes 1 \colon \mathbb{M}_{(n)}\to\mathbb{M}_{(n)}\otimes\mathbb{M}_{(n)}$ are homotopic. So the upper and lower square of the diagram commute up to homotopy, and commutes in the level of $K$-groups.

Second, we show the vertical map ${\rm id}\otimes 1\otimes {\rm id} \colon \mathcal{O}_{F_r}\otimes\mathbb{M}_{(n)}\to\mathcal{O}_{F_r}\otimes\mathbb{M}_{(n)}\otimes\mathbb{M}_{(n)}$ induces an isomorphism of the K-groups. One has an isomorphism $\psi : \mathbb{M}_{(n)}\to\mathbb{M}_{(n)}\otimes\mathbb{M}_{(n)}$. By \cite[Theorem 2.2]{DW}, two maps $1\otimes {\rm id}$ and $\psi$ are homotopic.  So the map $K_0({\rm id}\otimes 1\otimes {\rm id})=K_0({\rm id}\otimes \psi)$ is an isomorphism.\\

Finally, we show $(1-[F_r])K^0(X)\otimes \mathbb{Z}_{(n)}=(1-[F_R])K^0(X)\otimes \mathbb{Z}_{(n)}$.
An exact sequence $0\to \mathcal{K}_{F_r}\otimes \mathbb{M}_{(n)}\to\mathcal{T}_{F_r}\otimes\mathbb{M}_{(n)}\to\mathcal{O}_{F_r}\otimes \mathbb{M}_{(n)}\to 0$ gives a $6$-term exact sequence, and we have the following exact sequence :
$$K_0(C(X))\otimes\mathbb{Z}_{(n)}\xrightarrow{(1-[F_r])\otimes 1}K_0(C(X))\otimes \mathbb{Z}_{(n)}\xrightarrow{K_0(\theta_{F_r}\otimes{\rm id})}K_0(\mathcal{O}_{F_r}\otimes \mathbb{M}_{(n)}).$$
So we have $\operatorname{Ker}K_0(\theta_{F_r}\otimes {\rm id})=(1-[F_r])K^0(X)\otimes \mathbb{Z}_{(n)}$. This gives the conclusion because the diagram below commutes by the above argument :
$$\xymatrix{
K_0(C(X))\otimes \mathbb{Z}_{(n)}\ar[rr]^{K_0({\rm id}\otimes \theta_{F_r})}\ar@{=}[d]&&K_0(\mathcal{O}_{F_r}\otimes \mathbb{M}_{(n)})\ar[d]^{K_0(\varphi)}\\
K_0(C(X))\otimes \mathbb{Z}_{(n)}\ar[rr]^{K_0({\rm id}\otimes \theta_{F_R})}&&K_0(\mathcal{O}_{F_R}\otimes \mathbb{M}_{(n)}).
}$$
 
\end{proof}

We define an equivalence relation $\sim_n$ in $\tilde{K}^0(X)\otimes \mathbb{Z}_{(n)}$.

\begin{dfn}
Let $a$ and $b$ be elements in $\tilde{K}^0(X)\otimes\mathbb{Z}_{(n)}$. Then $a\sim_n b$ if there exists $z\in\tilde{K}^0(X)\otimes\mathbb{Z}_{(n)}$ satisfying $(n+a)(1+z)=(n+b)$.
\end{dfn}

All elements of $\tilde{K}^0(X)\otimes\mathbb{Z}_{(n)} $ are nilpotent by \cite[Chap.II, Theorem 5.9]{K}.
The relation $\sim_n$ is well-defined because  $(1-z)$ has the inverse $\sum_{k=0}^{\infty}z^k$.
For a vector bundle $E$ of rank $m$, we denote $[\tilde{E}]\colon=[E]-m$.

\begin{lem}\label{equiv}
Let $X$ be a connected compact Hausdorff space, and let $F_r$ and $F_R$ be vector bundles of rank $nr+1$ and $nR+1$ respectively with $(n, r)=(n, R)=1$. 
If $(1-[F_r])K^0(X)\otimes \mathbb{Z}_{(n)}=(1-[F_R])K^0(X)\otimes \mathbb{Z}_{(n)}$, we have $[\tilde{F_r}]r^{-1}\sim_n [\tilde{F_R}]R^{-1}$.
\end{lem}
\begin{proof}
By assumption we have $h\in K^0(X)\otimes \mathbb{Z}_{(n)}$ satisfying $(nr+[\tilde{F_r}])h=(nR+[\tilde{F_R}])$. A split exact sequence $0\to\tilde{K}^0(X)\otimes\mathbb{Z}_{(n)}\to K^0(X)\otimes \mathbb{Z}_{(n)}\xrightarrow{{\rm ev}_{pt}}K^0(\{pt\})\otimes \mathbb{Z}_{(n)}\to 0$ yields $h-R/r\in\tilde{K}^0(X)\otimes\mathbb{Z}_{(n)}$. So we have $(n+[\tilde{F_r}]r^{-1})(1+\frac{r}{R}(h-R/r))=(n+[\tilde{F_R}])$.
\end{proof}
By Lemma \ref{inv} and Lemma \ref{equiv} , the map $I_n : \mathcal{O}(X)_n\ni [\mathcal{O}_{F_r}\otimes \mathbb{M}_{(n)}]\mapsto [[\tilde{F_r}]r^{-1}]\in \tilde{K}^0(X)\otimes \mathbb{Z}_{(n)}/\sim_n$ is well-defined.
\begin{lem}\label{big}
Let $X$ be a finite dimensional connected compact Hausdorff space. Then the map $I_n$ is surjective, and we have 
$$|[X, \operatorname{BAut}\mathcal{O}_{n+1}]|\geq|\mathcal{O}(X)_n|\geq |\tilde{K}^0(X)\otimes \mathbb{Z}_{(n)}/\sim_n|.$$
\end{lem}
\begin{proof}
Every element of $\tilde{K}^0(X)\otimes \mathbb{Z}_{(n)}$ is of the form $\frac{1}{r}\otimes x$ where $(n, r)=1$ and $x\in\tilde{K}^0(X)$.
By \cite[Section 9, Theorem 1.2]{H}, we have $R\in \mathbb{N}$ satisfying $\tilde{K}^0(X)=\{[\tilde{E}]\in\tilde{K}^0(X)\;\mid {\rm rank}\;E=nR+1\}$. So we have a vector bundle of rank $nrR+1$, $F_{rR}$ with $Rx=[\tilde{F}_{rR}]$. Therefore we have $I_n([\tilde{F}_{rR}])=[\frac{1}{r}\otimes x]$.
By \cite[Theorem 1.4]{D2}, one has $\mathcal{O}(X)_n\subset[X, \operatorname{BAut}\mathcal{O}_{n+1}]$. This proves the lemma.
\end{proof}
Let $R$ be a commutative algebra. A filtration of $R$ is a sequence of subalgebras $$\dotsm R_{k+1}\subset R_k\subset\dotsm\subset R_1=R$$ with $R_pR_q\subset R_{p+q}$.
Let $X$ be a finite CW-complex. Then the group $\tilde{K}^0(X)$ is a finitely generated commutative group by induction argument of attaching cells.  
The algebra $\tilde{K}^0(X)$ has a filtration
$$0=K^0_m(X)\subset\dotsm\subset K^0_1(X)=\tilde{K}^0(X)$$
by \cite[Section 2.1]{AH}.
Consider a sequence of $k$-skeletons $\{pt\}=X_0\subset X_1\subset\dotsm\subset X_m=X$. Then we define $K^0_k(X)$ by $\operatorname{Ker}(K^0(X)\to K^0(X_k))$.
If the cohomology groups of a finite CW-complex $X$ have no torsion,  one has $\operatorname{Tor}(K^0_k(X)/K^0_{k+1}(X), \mathbb{Z}_n)=0$ by \cite[Section 2.3]{AH} and \cite[Section 2.4]{AH}.  
Moreover Dadarlat shows in his proof of \cite[Theorem 5.3]{D1} that if the cohomology groups of the space $X$ have no $n$-torsion, one has $\operatorname{Tor}(K_k^0(X)/K^0_{k+1}(X), \mathbb{Z}_n)=0,\; m\geq k$.

The proof of the following lemma is the same as in the proof of  \cite[Lemma 5.2]{D1}.

\begin{lem}\label{tec}
Let $R$ be a filtered commutative ring with $0=R_m\subset R_{m-1}\dotsm \subset R_1=R$ and such that $R$ is finitely generated as an additive group. If $\operatorname{Tor}(R_k/R_{k+1}, \mathbb{Z}_n)=0$ for every $k$, we have $|(R\otimes\mathbb{Z}_{(n)})/\sim_{n}|\geq |R\otimes \mathbb{Z}_n|$.
\end{lem}
 
\begin{cor}\label{Syn}
Let $X$ be a finite CW-complex. Suppose $\operatorname{Tor}(H^*(X),\;\mathbb{Z}_n)=0$. Then we have 
$$|\tilde{K}^0(X)\otimes \mathbb{Z}_{(n)}/\sim_n|\geq |\tilde{K}^0(X)\otimes \mathbb{Z}_n|.$$
\end{cor}
We need the following proposition.
\begin{prop}[{\cite[Proposition 5.1]{D1}}]\label{tower}
Let $X$ be a finite CW-complex. Then we have
$$|\tilde{H}^{even}(X, \mathbb{Z}_n)|\geq|[X, \operatorname{BAut}\mathcal{O}_{n+1}]|,$$
where $\tilde{H}^{even}(X, \mathbb{Z}_n) \colon=\prod_{k\geq 1}H^{2k}(X, \mathbb{Z}_n)$.
\end{prop}
Now we show the following theorem.
\begin{thm}  
Let $X$ be a finite CW-complex. Suppose $\operatorname{Tor}(H^*(X),\;\mathbb{Z}_n)=0$. Then the map $I_n : \mathcal{O}(X)_n\to\tilde{K}^0(X)\otimes\mathbb{Z}_{(n)}/\sim_n$ is bijective, and we have
$$|[X, \operatorname{BAut}\mathcal{O}_{n+1}]|=|\mathcal{O}(X)_n|=|\tilde{H}^{even}(X, \mathbb{Z}_n)|.$$
\end{thm}

\begin{proof}
By Corollary \ref{Syn}, we have $|\tilde{K}^0(X)\otimes \mathbb{Z}_{(n)}/\sim_n|\geq |\tilde{K}^0(X)\otimes \mathbb{Z}_n|$. By Lemma \ref{big}, we have $|[X, \operatorname{BAut}\mathcal{O}_{n+1}]|\geq |\tilde{K}^0(X)\otimes \mathbb{Z}_{(n)}/\sim_n|$. From  Proposition \ref{tower}, we have $|\tilde{H}^{even}(X, \mathbb{Z}_n)|\geq|[X, \operatorname{BAut}\mathcal{O}_{n+1}]|,$ and  Theorem \ref{DG} yields 
$$|[X, \operatorname{BAut}\mathcal{O}_{n+1}]|=|\mathcal{O}(X)_n|=|\tilde{H}^{even}(X, \mathbb{Z}_n)|.$$

\end{proof}

\end{document}